\documentclass[a4paper,11pt,reqno,twoside]{amsart}

\usepackage{amsmath}
\usepackage{amsfonts}
\usepackage{amssymb}
\usepackage{amsthm}
\usepackage{color}
\usepackage{ifpdf}
\usepackage{array}
\usepackage{url}
\usepackage{multirow,bigdelim}

\numberwithin{equation}{section}

\newtheorem{theorem}{Theorem}
\newtheorem{lemma}{Lemma}
\newtheorem{proposition}{Proposition}

\newcommand{\comment}[1]{}
\title[Nearest neighbor spacing distributions for zeros]%
      {Nearest neighbor spacing distributions for zeros of the real or imaginary part of the Riemann xi-function on vertical lines} 
\author[M. Suzuki]{Masatoshi Suzuki}
\date{Version of \today}
\AtBeginDocument{%
\begin{abstract}
We show that the density functions of nearest neighbor spacing distributions for zeros 
of the real or imaginary part of the Riemann xi-function on vertical lines 
are described by the $M$-function which is appeared in value distributions 
of the logarithmic derivative of the Riemann zeta-function on vertical lines.  
\end{abstract}
\maketitle
}
\begin{document}

\section{Introduction}

Let $s=\sigma+it$ ($i=\sqrt{-1}$) be a complex variable, 
$\zeta(s)$ be the Riemann zeta-function, and 
\[
\xi(s)=\frac{1}{2}s(s-1)\pi^{-s/2}\Gamma\left(\frac{s}{2}\right)\zeta(s)
\]
be the Riemann xi-function, which is an entire function 
satisfying functional equations $\xi(s)=\xi(1-s)$ and $\xi(\bar{s})=\overline{\xi(s)}$. 
In this paper, we discuss the distributions of zeros of entire functions 
\begin{equation} \label{AB}
A_\omega(s) :=  \frac{1}{2}(\xi(s+\omega)+\xi(s-\omega)), \quad B_\omega(s) := \frac{i}{2}(\xi(s+\omega)-\xi(s-\omega))
\end{equation}
having a positive real parameter $\omega$ 
in consideration of the following two relations with the zeros of $\xi(s)$. 
Firstly, the zeros of $A_\omega(s)$ and $B_\omega(s)$ on the line $\sigma=1/2$ 
coincide respectively with the zeros of the real and imaginary parts of $\xi(s)$ 
on the line $\sigma=1/2+\omega$, 
because we have 
\begin{equation} \label{0807_3}
\aligned
{\rm Re}\,\xi(\tfrac{1}{2}+\omega+it) = A_\omega(\tfrac{1}{2}+it), \quad 
{\rm Im}\,\xi(\tfrac{1}{2}+\omega+it) = -B_\omega(\tfrac{1}{2}+it)
\endaligned
\end{equation}
by functional equations of $\xi(s)$. 
Secondly, for small $\omega>0$, the zeros of $A_\omega(s)$ and $B_\omega(s)$ 
(locally) approximate the zeros of $\xi(s)$ and $\xi'(s)$ respectively, 
because of asymptotic relations 
\[
A_\omega(s) = \xi(s) + O(\omega^2),  \quad 
B_\omega(s) = i\omega \cdot \xi'(s) + O(\omega^3) \quad (\omega \to 0^+)
\]
on compact subsets of $\mathbb C$.

The functional equations of $\xi(s)$ deduce that $A_\omega(s)$ and $B_\omega(s)$ satisfy
\[
A_\omega(s)=A_\omega(1-s), \quad B_\omega(s)=-B_\omega(1-s) 
\]
and take real values on the critical line $\sigma=1/2$. 
It is known that all zeros of $A_\omega(s)$ and $B_\omega(s)$ are simple zeros lying on the critical line 
if $\omega \geq 1/2$. 
This holds also for $0<\omega<1/2$ 
if we assume the Riemann Hypothesis (RH) for $\xi(s)$ (\cite[Theorem 2.1]{Jeff05}), 
or unconditionally, except for a set of zeros up to height $T$ of
cardinality $\ll T^{1-a\omega}(\log T)^2$ for any $a < 1$ (\cite[Theorem 1 and 2]{Li09}). 
In this sense, the horizontal distributions of the zeros of $A_\omega(s)$ and $B_\omega(s)$ are understood well. 
Therefore we turn interest to their vertical distributions in what follows. 
\medskip

Let $X_\omega(s)$ be $A_\omega(s)$ or $B_\omega(s)$. 
We arrange the zero $\rho=\beta+i\gamma$ of $X_\omega(s)$ with $\gamma>0$ 
in a sequence $\rho_n=\beta_n+i\gamma_n$ so that $\gamma_{n+1} \geq \gamma_n$. 
Then the distribution of spacings of the normalized imaginary parts 
\begin{equation} \label{0814_2}
\gamma_n^{(1)}:=  \frac{\gamma_n}{2\pi}\log\frac{\gamma_n}{2\pi e}
\end{equation}
converges to a limiting distribution of equal spacings of length one. 
This fact is proved in Lagarias \cite[Theorem 4.1]{Jeff05} by assuming RH 
if $0<\omega<1/2$ and in Li \cite[Theorem 1]{Li09} unconditionally.  
The above result on the normalized imaginary parts is contrast to the Montgomery--Odlyzko 
conjecture and the GUE conjecture  
which assert that the distribution of the normalized imaginary parts of the zeros of $\xi(s)$ 
obeys the distribution of eigenvalues of random hermitian matrices from the Gaussian Unitary Ensemble (GUE). 
Therefore, one may consider that the zeros of $A_\omega(s)$ and $B_\omega(s)$ 
are insignificant objects at least from the viewpoint of their vertical distributions. 

However, interestingly enough, it will be proved that the {\it second normalization} 
of the imaginary parts defined by
\begin{equation} \label{0731_1}
\gamma_{n}^{(2)}:= \Bigl( \frac{\gamma_n}{2\pi}\log\frac{\gamma_n}{2\pi e} - n \Bigr)
\varrho_\omega^{-1/2}
\frac{1}{2\pi}\log\frac{\gamma_n}{2\pi e}
\end{equation}
have a remarkable distribution which is related to the Euler product of the Riemann zeta-function, 
where 
\[
\varrho_\omega := \frac{1}{2\pi^2}\sum_{n=1}^{\infty}\frac{\Lambda(n)^2}{n^{1+2\omega}} 
\]
for the von Mangoldt function $\Lambda(n)$ 
and the series converges absolutely for $\omega>0$. 
\medskip

In order to state the main theorem, we recall a result on 
the value distributions of the logarithmic derivative of the Riemann zeta-function 
on vertical lines.   
For every $\sigma>1/2$, there exists a non-negative real valued $C^\infty$-function $M_\sigma(z)$ on $\mathbb C$ such that 
$(2\pi)^{-1}\int_{\mathbb C} M_\sigma(z)\, dz=1$ and the formula 
\begin{equation} \label{formula_IM}
\lim_{T\to\infty}\frac{1}{2T}\int_{-T}^{T}\Phi\left( \frac{\zeta'}{\zeta}(\sigma+it) \right) \, dt
= \frac{1}{2\pi} \int_{\mathbb C} M_{\sigma}(z)\Phi(z)\,dz
\end{equation}
holds for any continuous bounded function $\Phi(z)$ on $\mathbb C$ 
or the characteristic function of either a compact subset of $\mathbb C$ or the complement of such a subset. 
We call $M_\sigma(z)$ the {\it $M$-function} according to \cite{Ihara01}. 
The above formula was obtained by 
Kampen-Wintner~\cite{vKW37}, Kershner--Wintner~\cite{KW37}, Guo~\cite{Guo}, Ihara~\cite{Ihara01} and Ihara-Matsumoto~\cite{IM}
(see Appendix for a construction of $M_\sigma(z)$ and its historical details). 
If $\sigma>1$, formula \eqref{formula_IM} holds for any continuous function $\Phi(z)$ on $\mathbb C$. 

Using the $M$-function, we define the {\it $m$-function} by
\begin{equation} \label{formula_IM2}
m_\sigma(u)=\int_{-\infty}^{\infty}M_{\sigma}(u+iv) \, dv 
\end{equation}
on the real line. 
This is well-defined because $M_\sigma(z)$ is of rapid decay (\cite[Theorem 2]{Ihara01}). 

\noindent
Reflecting the Euler product formula of the Riemann zeta-function, 
the Fourier transform $\tilde{M}_\sigma(z)$ has an Euler product formula 
$\tilde{M}_\sigma(z)=\prod_{p} \tilde{M}_{\sigma,p}(z)$ 
whose local factors $\tilde{M}_{\sigma,p}(z)$ 
are some arithmetic Dirichlet series in $\sigma$, 
where $p$ runs over all prime numbers (see Appendix). 
Therefore, the Fourier transform of $m$-function also has an Euler product, 
since  
\[
\tilde{m}_\sigma(x)
=\frac{1}{2\pi}\int_{-\infty}^{\infty}m_\sigma(u)e^{i xu} du
= \frac{1}{2\pi}\int_{\mathbb C}M_\sigma(u+iv)e^{i xu} dudv 
= \tilde{M}_\sigma(x). 
\]

Now the main result is stated  as follows. 
\begin{theorem}\label{thm_01} 
Let $X_\omega(s)$ be $A_\omega(s)$ or $B_\omega(s)$ for given $\omega>0$, 
and let $\gamma_n^{(2)}$ be the secondary normalized imaginary parts of 
the zeros of $X_\omega(s)$ 
defined in \eqref{0731_1}. 
Then the formula 
\begin{equation} \label{0731_2}
\lim_{T \to \infty} \frac{1}{N_\omega(T)}\sum_{0<\gamma_n \leq T}\phi(\gamma_{n+1}^{(2)}-\gamma_{n}^{(2)})
= \frac{1}{2\pi}\int_{-\infty}^{\infty} \pi\varrho_{\omega}^{1/2}m_{\frac{1}{2}+\omega}(\pi \varrho_{\omega}^{1/2}u)\phi(u) \, du
\end{equation}
holds for any bounded function $\phi \in C^1({\mathbb R})$  
such that $\phi'(x) \ll 1$ for $|x| \leq 1$, 
$\phi'(x) \ll x^{-2}$ for $|x| \geq 1$ and 
$u \mapsto \tfrac{d}{du}\phi\bigl({\rm Re}\tfrac{\zeta'}{\zeta}(\tfrac{1}{2}+\omega+iu)\bigr)$ 
is bounded on $\mathbb R$, 
where $N_\omega(T)$ is the number of zeros of $X_\omega(s)$ with $0 < t \leq T$. 
\end{theorem}
The limit behavior of the integrand of the right-hand side of \eqref{0731_2} 
as $\omega \to 0^+$ is obtained 
as follows by using a result of \cite{Ihara02}. 
\begin{theorem}\label{thm_02} We have 
\[
\frac{1}{2\pi}\lim_{\omega \to 0^+} \pi\varrho_{\omega}^{1/2}m_{\frac{1}{2}+\omega}(\pi \varrho_{\omega}^{1/2}u) 
= \frac{1}{\sqrt{2\pi}} \exp\left( -\frac{u^2}{2} \right).
\]
\end{theorem}
Note that the above two theorems are unconditional.

We now make a consideration on a significance of Theorem \ref{thm_01} 
under RH if $0 < \omega < 1/2$. 
In this case, all zeros of $X_\omega(s)$ are simple zeros lying on the critical line and  
\begin{equation}\label{0425_1}
N_\omega(T) 
= \frac{T}{2\pi}\log\frac{T}{2\pi e}+S_\omega(T)+ \frac{7+2\omega}{8} + O\left( \frac{1}{T} \right),
\end{equation}
for $T \geq 2$ (\cite[Theorem 3.1]{Jeff05}), where 
\[
S_\omega(t) 
= \frac{1}{\pi} \arg \zeta(\tfrac{1}{2}+\omega+it) 
\]
is a $C^\infty$-function on the real line obtained by continuous variation along the straight lines joining 
$2$, $2+it$ and $1/2+\omega+it$, starting with the value $0$. 
By the simplicity of zeros, \eqref{0814_2} and \eqref{0425_1}, we have 
\[
1 = N_\omega(\gamma_{n+1})-N_\omega(\gamma_n) 
= \gamma_{n+1}^{(1)} - \gamma_{n}^{(1)} + S_\omega(\gamma_{n+1}) - S_\omega(\gamma_n) + O\left( \frac{1}{\gamma_{n}} \right), 
\]
and thus
\begin{equation} \label{0814_3}
\gamma_{n+1}^{(1)} - \gamma_{n}^{(1)} - 1 
=  - \Bigl( S_\omega(\gamma_{n+1}) - S_\omega(\gamma_n) \Bigr) + O\left( \frac{1}{\gamma_{n}} \right).
\end{equation}
Given this formula, $\gamma_{n+1}^{(1)} - \gamma_{n}^{(1)} \to 1$ means that 
the contribution of $S_\omega(\gamma_{n+1}) - S_\omega(\gamma_n)$ is smaller than $1$ for any fixed $\omega>0$. 
In other words, the distribution of spacings of the normalized zeros of $X_\omega(s)$
is dominated by the gamma functor of $\zeta(s)$ only. 

On the other hand, it is known that a subtle behavior of the zeros of 
$\zeta(s)$ such as the Montgomery--Odlyzko conjecture 
is caused by the function $S(t)$, 
which is obtained by $S(t)=\lim_{\omega \to 0^+}S_\omega(t)$ if $t$ is not the ordinate of a zero of $\zeta(s)$, 
and $S(t)=\tfrac{1}{2}\lim_{\delta\to0^+}(S(t+\delta)+S(t-\delta))$ 
if $t$ is not the ordinate of a zero of $\zeta(s)$. 

Therefore, from the discussion above, Theorem \ref{thm_01} shows that 
the second normalization \eqref{0731_1} detects an effect  
of the arithmetic part $S_\omega(T)$ of the counting function $N_\omega(T)$. 
An Euler product formula of $\tilde{m}_\sigma(u)$ is a supporting evidence of this observation.
\bigskip

A motivation of this work was L. Weng's question to the author. 
In 2013, he and D. Zagier proved that all high-rank zeta functions for elliptic curves $E$ 
defined over a finite field satisfy an analogue of the Riemann Hypothesis (\cite{WZ}). 
Then he considered a distribution of the zeros of high-rank zeta functions for $E$ 
when the rank is varied 
and observed that 
the dominant term is very simple but 
the second dominant term is related to the Sato-Tate measure. 
His question to the author was what an analogue of his observation to the number field case is (\cite{We1},  
where he considered another version of \eqref{0731_1} 
but it is simplified in \cite{We2} as compatible with \eqref{0731_1}).    
For the rational number field $\mathbb Q$, high-rank zeta functions $\hat{\zeta}_{{\mathbb Q},n}(s)$ 
are expressed as linear combinations of products of the Riemann zeta-function and rational functions. 
The rank one case is $\hat{\zeta}_{{\mathbb Q}, 1}(s)=\hat{\zeta}(s)$. 
The rank two case is 
\[
s(2s-1)(2s-2)\hat{\zeta}_{{\Bbb Q},2}(s) 
= \xi(2s) - \xi(2s-1) 
= B_{1/2}(2s-\tfrac{1}{2}). 
\]
Therefore, the second dominant term of the distribution of the zeros is described by $m_1(x)$. 
The rank three case is
\[
\aligned
3s(3s-1)& (3s-2)(3s-3)\hat{\zeta}_{{\Bbb Q},3}(s) = X(s)+X(1-s), \\
X(s) & =\Bigl(3(2\xi(2)-1)s-4\xi(2) + 3\Bigr)\xi(3s) -  \xi(3s-1).  
\endaligned
\]
This looks similar to $A_{1}(3s-1)= \xi(3s) + \xi(3s-2)$ in a sense. 
Therefore, it is expected that 
the second dominant term of the distribution of the zeros is described by $m_{3/2}(x)$ 
up to a small correction. 
\medskip

This paper is organized as follows. 
In Section 2, we prepare some lemmas necessary for the proof of Theorem \ref{thm_01}. 
In Section 3, we prove Theorem \ref{thm_01} under RH at first 
for the simplicity of argument. 
Then we prove Theorem \ref{thm_01} unconditionally and prove Theorem \ref{thm_02}. 
In Section 4, we several comments and remarks on subjects of the paper. 
Finally, we provide a review of construction, basic properties and history of the $M$-function as an appendix. 

\section{Preliminaries} 

Let $\omega>0$. We will assume RH  if $0 < \omega < 1/2$ throughout this section. 
Then the imaginary parts of $A_\omega(s)$ and $B_\omega(s)$ are enumerated as 
\[
\cdots <\gamma_{-1}(B_\omega)<\gamma_{-1}(A_\omega)<\gamma_{0}(B_\omega)=0
<\gamma_{1}(A_\omega)<\gamma_{1}(B_\omega)<\gamma_{2}(A_\omega)<\gamma_{2}(B_\omega)<\cdots.
\]
with $\gamma_{-n}(A_\omega)=-\gamma_{n}(A_\omega)$ 
and $\gamma_{-n}(B_\omega)=-\gamma_{n}(B_\omega)$ for $n \geq 1$. 
We denote by $\gamma_n$ the $n$th imaginary part $\gamma_n(A_\omega)$ or $\gamma_n(B_\omega)$ when $n \ge 1$. 

\begin{lemma} We have  
\begin{equation} \label{0529_1}
\gamma_{n} = \frac{2\pi n}{\log n}\left( 1 + O\left(\frac{\log\log n}{\log n}\right) \right),
\end{equation}
\begin{equation} \label{0808_1}
\log \frac{\gamma_n}{2\pi} 
 = \log n
\left(1 + O\left(\frac{\log\log n}{\log n} \right) \right).
\end{equation}
These formulas are unconditional. 
\end{lemma}
\noindent
{\bf Remark.} It is claimed that  
\[
\gamma_{n} = \frac{2\pi n}{\log n}\left( 1 + O\left(\frac{1}{\log n}\right) \right) 
\]
in \cite[p.171]{Jeff05} standing on \eqref{0425_1} and $S_\omega(t)=O(\log t)$. 
However, the author do not know how to exclude the factor $\log\log n$ from \eqref{0529_1}. 
\begin{proof}
Suppose that $\gamma_n=\gamma_n(A_\omega)$. 
We have $S_\omega(T)=O(\log T)$ unconditionally 
as well as \cite[Theorem 9.4]{Tit}, 
where the implied constant does not depend on $\omega$. 
Therefore, 
\[
n = N_\omega(\gamma_n) = \frac{\gamma_n}{2\pi}\log\frac{\gamma_n}{2\pi e}\left(1+O\left( \frac{1}{\gamma_n} \right) \right)
\]
by the simplicity of zeros. This implies 
\[
\log n =  
\log \frac{\gamma_n}{2\pi}\left(1 +\frac{\log \log\frac{\gamma_n}{2\pi e}}{\log \frac{\gamma_n}{2\pi}} 
+ O\left( \frac{1}{\gamma_n \log \gamma_n} \right) \right).
\]
Taking the quotient of these two equalities, 
\[
\aligned
\frac{2\pi n}{\log{n}}
&= \gamma_n \, \frac{ \log\frac{\gamma_n}{2\pi}\left(1+O\left( \frac{1}{\gamma_n} \right) \right) 
-\left(1+O\left( \frac{1}{\gamma_n} \right) \right) 
}
{\log \frac{\gamma_n}{2\pi}\left(1 +\frac{\log \log\frac{\gamma_n}{2\pi e}}{\log \frac{\gamma_n}{2\pi}} 
+ O\left( \frac{1}{\gamma_n \log \gamma_n} \right) \right)} \\
&= \gamma_n \, 
\left(\frac{1+O\left( \frac{1}{\gamma_n} \right) 
}
{1 +\frac{\log \log\frac{\gamma_n}{2\pi e}}{\log \frac{\gamma_n}{2\pi}} + O\left( \frac{1}{\gamma_n \log \gamma_n} \right)}
+O\left( \frac{1}{\log\gamma_n} \right) 
\right) \\
&=\gamma_n \, \left(
1+O\left( \frac{\log\log\gamma_n}{\log\gamma_n} \right) 
\right).
\endaligned
\]
Therefore,
\[
\gamma_n
= \frac{2\pi n}{\log{n}}
\left(
1+O\left( \frac{\log\log\gamma_n}{\log\gamma_n} \right) 
\right). 
\]
In particular, $n/(\log n) \ll \gamma_n$ by $\gamma_n \to \infty$. 
Hence we obtain  \eqref{0529_1} by 
$\log\log\gamma_n/\log\gamma_n \ll \log\log (n/(\log n))/\log (n/(\log n)) \ll \log\log n/\log n$. 
By \eqref{0529_1}, we have 
\[
\log \frac{\gamma_n}{2\pi} 
 = \log{n}\left(1-\frac{\log\log n}{\log n}\right)\left(1 +O\left(\frac{\log\log n}{\log n}\right)\right) 
= \log n
\left(1 + O\left(\frac{\log\log n}{\log n} \right) \right).
\]
This is nothing but \eqref{0808_1}. 
The case of $\gamma_n=\gamma_n(B_\omega)$ is proved in a similar way. 
\end{proof}

\begin{lemma} \label{0808_3}
The gaps $\gamma_{n+1} - \gamma_{n}$ tend to $0$ as $n \to \infty$. 
\end{lemma}
\begin{proof} 
We show that $S_\omega(t)=o(\log t)$ holds for any fixed $\omega>0$, 
because it implies Lemma \ref{0808_3} by \eqref{0425_1}. 
 We have 
\[
\log \zeta(\tfrac{1}{2}+\omega+it)
\ll
\begin{cases}
1 & \text{if $\omega > 1/2$}, \\[2pt]
\log\log t & \text{if $\omega = 1/2$}, \\[2pt]
\log\log\log t & \text{if $\omega = 1/2$ under RH}, \\[2pt]
\displaystyle{\frac{(\log t)^{1-2\omega}}{\log\log t} } & \text{if $0 < \omega < 1/2$ under RH}.
\end{cases}
\]
for large $t>0$, 
where the first line is a consequence of the absolute convergence of the Dirichlet series of $\log\zeta(s)$, 
the second line is shown in \cite[Theorem 6.7]{MV} 
and the other cases are shown in \cite[Theorem 14.5, \S14.33]{Tit}. 
These estimates imply $S_\omega(t)=o(\log t)$, since $S_\omega(t) \ll |\log \zeta(1/2+\omega+it)|$.  
\end{proof}

\begin{lemma} We have
\begin{equation} \label{0809_0}
\frac{S_{\omega}(\gamma_{n+1}) - S_{\omega}(\gamma_{n})}{\gamma_{n+1}-\gamma_{n}}=O(E_{1,\omega}(\gamma_n))
\end{equation}
with
\begin{equation} \label{0809_1}
E_{1,\omega}(t)
=
\begin{cases}
\displaystyle{1} & \text{if $\omega > 1/2$}, \\[10pt]
\displaystyle{ \frac{\log t}{\log\log t} } & \text{if $\omega = 1/2$}, \\[10pt]
\displaystyle{ \log\log t } & \text{if $\omega = 1/2$ under RH}, \\[10pt]
\displaystyle{ (\log t)^{1-2\omega} } & \text{if $0 < \omega < 1/2$ under RH}.
\end{cases}
\end{equation}
\end{lemma}
\begin{proof}
We have 
$\pi S_{\omega}^\prime(t) = {\rm Re} (\zeta'/\zeta)(1/2+\omega+it)$ 
by the definition of $S_\omega(t)$, 
since $\zeta(s)$ has no zeros in ${\rm Re}(s)>1/2$ by RH. 
Therefore, 
\[
\pi \left|
\frac{S_{\omega}(\gamma_{n+1}) - S_{\omega}(\gamma_{n})}{\gamma_{n+1}-\gamma_{n}}
\right|
 \leq  \left| {\rm Re}\left\{ \frac{\zeta'}{\zeta}(\tfrac{1}{2}+\omega+i\gamma) \right\} \right|
 \leq \left| \frac{\zeta'}{\zeta}(\tfrac{1}{2}+\omega+i\gamma) \right| 
\]
for some $\gamma_n < \xi < \gamma_{n+1}$ by Lemma \ref{0808_3} and the mean value theorem. 
On the right-hand side, we have
\begin{equation} \label{0809_7}
\frac{\zeta'}{\zeta}(\tfrac{1}{2}+\omega+i\xi)
\ll 
E_{1,\omega}(\xi),
\end{equation}
where the first line of \eqref{0809_1} is a consequence of the absolute convergence of the Dirichlet series of $(\zeta'/\zeta)(s)$, 
the second line of \eqref{0809_1} is shown in \cite[(5.14.7)]{Tit} 
and the other cases of \eqref{0809_1} are shown in \cite[\S14.33]{Tit}. 
These estimates imply \eqref{0809_0}, since 
$ \log \xi < \log {\gamma_{n+1}} = \log {\gamma_{n}} +O(\gamma_n^{-1}) $ 
by Lemma \ref{0808_3}. 
\end{proof}

\begin{lemma} \label{0812_1} We have 
\begin{equation} \label{0809_2}
\frac{\gamma_{n+1} - \gamma_{n}}{2\pi}\log\frac{\gamma_n}{2\pi e} = 
1 + O(E_{2,\omega}(\gamma_n)),  
\end{equation}
where $E_{2,\omega}(t) = E_{1,\omega}(t)/\log t$ for the function $E_{1,\omega}(t)$ of \eqref{0809_1}. 
\end{lemma}
\begin{proof} 
We have 
\[
N_{\omega}(t+h) -N_{\omega}(t) 
= \frac{h}{2\pi}\log\frac{t}{2\pi}
+S_{\omega}(t+h) - S_{\omega}(t) 
+O\left(\frac{1}{t+1}\right)
\]
for $0 \leq h \leq 1$ and $t \geq 2$ 
as well as the proof of \cite[Theorem 4.1]{Jeff05}, where the implied constant does not depend on $h$. 
Applying this to $t=\gamma_n$ and $h=\gamma_{n+1}-\gamma_{n}$ together with Lemma \ref{0808_3}, we get
\[
1=  N_{\omega}(\gamma_{n+1}) - N_{\omega}(\gamma_{n}) 
= \frac{\gamma_{n+1}-\gamma_{n}}{2\pi}\log\frac{\gamma_{n}}{2\pi e}
+ S_{\omega}(\gamma_{n+1}) - S_{\omega}(\gamma_{n})
+O\left(\frac{1}{\gamma_{n}}\right) 
\]
for large $n$. This implies 
\[
(\gamma_{n+1}-\gamma_{n})\frac{1}{2\pi}\log\frac{\gamma_{n}}{2\pi e}
\left( 1 + O\left( 
\frac{1}{\log\gamma_{n}}\left| \frac{S_{\omega}(\gamma_{n+1}) - S_{\omega}(\gamma_{n})}{\gamma_{n+1}-\gamma_{n}} \right|
\right)
\right)
= 1+O\left(\frac{1}{\gamma_{n}}\right).
\]
Applying \eqref{0809_0} to the left-hand side, we obtain \eqref{0809_2}.
\end{proof}
\begin{lemma} 
Assume that $f(t)$ belongs to $C^1({\mathbb R})$ and $f'(t)$ is bounded on $\mathbb R$. 
Then, 
\begin{equation} \label{0810_1}
\frac{1}{\gamma_{N}}\sum_{n=1}^{N-1} f(\gamma_n)(\gamma_{n+1}-\gamma_{n})
 = 
\frac{1}{\gamma_{N}}\int_{0}^{\gamma_{N}}f(t)\,dt
+ O\left( \frac{1}{\log \gamma_N} \right)
\end{equation}
holds for large $N>0$. 
\end{lemma}
\begin{proof}
We have
\[
\frac{1}{\gamma_{N}}\sum_{n=1}^{N-1} f(\gamma_n)(\gamma_{n+1}-\gamma_{n})
= 
\frac{1}{\gamma_{N}}\int_{\gamma_{1}}^{\gamma_{N}}f(t)\,dt
+\frac{1}{\gamma_{N}}\sum_{n=1}^{N-1} \int_{\gamma_{n}}^{\gamma_{n+1}}(f(\gamma_n)-f(t))\,dt
+ O\left(\frac{1}{\gamma_N} \right).
\]
The second sum on the right-hand side is estimated as 
\[
\aligned
\left| \sum_{n=1}^{N-1} \int_{\gamma_{n}}^{\gamma_{n+1}}(f(\gamma_n)-f(t))\,dt \right|
& \leq \sum_{n=1}^{N-1} \max_{\gamma_{n}\leq \xi \leq \gamma_{n+1}}|f'(\xi)|\int_{\gamma_{n}}^{\gamma_{n+1}}(t-\gamma_n)\,dt \\
& \leq \frac{1}{2}\max_{\gamma_{1}\leq t <\infty}|f'(t)| \sum_{n=1}^{N-1} (\gamma_{n+1}-\gamma_n)^2 .
\endaligned
\]
Here the sum on the right-hand side is estimated as 
\[
\sum_{n=1}^{N-1} (\gamma_{n+1}-\gamma_n)^2 
\ll \sum_{n=1}^{N-1} \frac{1}{\log \gamma_n}, 
\]
since $\gamma_{n+1}-\gamma_n \ll (\log \gamma_n)^{-1}$ by \eqref{0809_2}. 
Using the Stietjes integral and integration by parts, we have 
\[
\sum_{n=1}^{N-1} \frac{1}{\log \gamma_n}
 \ll \int_{\gamma_1}^{\gamma_N}\frac{dN_\omega(t)}{(\log t)^2}
 \ll \int_{\gamma_1}^{\gamma_N}\frac{dt}{\log t} 
 \ll \frac{\gamma_N}{\log \gamma_N}.
\]
Hence we obtain \eqref{0810_1}. 
\end{proof}

\section{Proofs of results} \label{section_proofs}

At first, we prove Theorem \ref{thm_01} 
assuming RH if $0<\omega<1/2$ 
after preparing two propositions 
standing on results in the previous section. 

\begin{proposition} \label{prop_1}
Assume that $f(t)$ belongs to $C^1({\mathbb R})$ and is bounded on $\mathbb R$. 
Then, 
\begin{equation} \label{0810_3}
\frac{1}{N_\omega(T)} 
\sum_{0<\gamma_n \leq T} f(\gamma_n) 
= \frac{1}{T}\int_{0}^{T} f(t)\,dt 
+ O\left( E_{2,\omega}(T)\right)
\end{equation}
holds for large $T>0$, 
where $E_{2,\omega}(t) = E_{1,\omega}(t)/\log t$ for the function $E_{1,\omega}(t)$ of \eqref{0809_1}. 
\end{proposition}
\begin{proof} 
It is sufficient to show that the lef-hand side of \eqref{0810_1} 
is equal to the left-hand side of \eqref{0810_3} up to a reasonable error terms. 
We have 
\[
\aligned
\frac{1}{\gamma_{N}} 
 \sum_{n=1}^{N-1}  f(\gamma_n)(\gamma_{n+1}-\gamma_{n}) 
 & = 
\frac{1}{\frac{\gamma_{N}}{2\pi}\log\frac{\gamma_{N}}{2\pi}} 
\sum_{n=1}^{N-1}f(\gamma_n)  \frac{\gamma_{n+1}-\gamma_{n}}{2\pi}\log\frac{\gamma_{N}}{2\pi} \\
& = 
 \frac{1}{\frac{\gamma_{N}}{2\pi}\log\frac{\gamma_{N}}{2\pi}} 
 \sum_{n=1}^{N-1}f(\gamma_n)
 \frac{\gamma_{n+1}-\gamma_{n}}{2\pi}\log\frac{\gamma_{n}}{2\pi} \\ 
& \quad  
+ 
 \frac{1}{\frac{\gamma_{N}}{2\pi}\log\frac{\gamma_{N}}{2\pi}} 
 \sum_{n=1}^{N-1}f(\gamma_n)
 \frac{\gamma_{n+1}-\gamma_{n}}{2\pi} \log\frac{\gamma_n}{2\pi} \left( \frac{\log\frac{\gamma_N}{2\pi}}{\log\frac{\gamma_n}{2\pi}} -1 \right) 
\\ 
& = S_1+S_2, 
\endaligned
\]
say. First we consider $S_1$. We have 
\[
\left|
S_1 - \frac{1}{\frac{\gamma_{N}}{2\pi}\log\frac{\gamma_{N}}{2\pi}} 
 \sum_{n=1}^{N-1}f(\gamma_n)
\right|
\ll
 \frac{1}{\frac{\gamma_{N}}{2\pi}\log\frac{\gamma_{N}}{2\pi}} 
 \sum_{n=1}^{N-1}
 \left| \frac{\gamma_{n+1}-\gamma_{n}}{2\pi} \log\frac{\gamma_n}{2\pi} - 1 \right|. 
\]
For the sum on the right-hand side,  
\[
\sum_{n=1}^{N-1}
 \left| \frac{\gamma_{n+1}-\gamma_{n}}{2\pi} \log\frac{\gamma_n}{2\pi} - 1 \right|
 \ll 
\sum_{n=1}^{N-1}E_{2,\omega}(\gamma_n) 
 \ll 
\int_{\gamma_1}^{\gamma_N} E_{2,\omega}(t)dN_\omega(t)
\]
by \eqref{0809_2} and the Stietjes integral.  Here 
\[
\int_{\gamma_1}^{\gamma_N} E_{2,\omega}(t)dN_\omega(t) 
\ll \int_{\gamma_1}^{\gamma_N} E_{2,\omega}(t)(\log t)dt 
\ll \gamma_N \log \gamma_N E_{2,\omega}(\gamma_N)
\]
by integration by parts. Hence 
\[
\left|
S_1 - \frac{1}{\frac{\gamma_{N}}{2\pi}\log\frac{\gamma_{N}}{2\pi}} 
 \sum_{n=1}^{N-1}f(\gamma_n)
\right|
\ll E_{2,\omega}(\gamma_N).
\]
Next we consider $S_2$. We have 
\[
|S_2|  \ll
 \frac{1}{\frac{\gamma_{N}}{2\pi}\log\frac{\gamma_{N}}{2\pi}} 
 \sum_{n=1}^{N-1} \left( \frac{\log\frac{\gamma_N}{2\pi}}{\log\frac{\gamma_n}{2\pi}} -1 \right) 
\]
by \eqref{0809_2}. Using the partial summation for the sum on the right-hand side, 
\[
\aligned
\frac{1}{\frac{\gamma_{N}}{2\pi}\log\frac{\gamma_{N}}{2\pi}} \sum_{n=1}^{N-1}
\left( \frac{\log\frac{\gamma_N}{2\pi}}{\log\frac{\gamma_n}{2\pi}} -1 \right)
& = \frac{2\pi}{\gamma_N} \int_{\gamma_1}^{\gamma_N}
\Bigl(\sum_{0<\gamma_n \leq x}1 \Bigr) \frac{1}{x(\log\frac{x}{2\pi})^2}\, dx+ O\left(\frac{1}{\gamma_N}\right)\\
& \ll  \frac{1}{\gamma_N}  \int_{\gamma_1}^{\gamma_N} x\log x \cdot \frac{1}{x(\log x)^2}\, dx+ O\left(\frac{1}{\gamma_N}\right)
\ll \frac{1}{\log \gamma_N}.
\endaligned
\]
From the above argument, we obtain
\[
\frac{1}{\gamma_{N}} 
 \sum_{n=1}^{N-1}  f(\gamma_n)(\gamma_{n+1}-\gamma_{n}) 
= 
 \frac{1}{\frac{\gamma_{N}}{2\pi}\log\frac{\gamma_{N}}{2\pi}} 
 \sum_{n=1}^{N-1}f(\gamma_n)
 +O(E_{2\omega}(\gamma_N)) , 
\]
since $(\log t)^{-1} \ll E_{2,\omega}(t)$ for every $\omega>0$. 
Combining this with \eqref{0810_1} and 
\[
\frac{\gamma_{N}}{2\pi}\log\frac{\gamma_{N}}{2\pi}
=N_\omega(\gamma_N)\left(1+O\left(\frac{1}{\gamma_N}\right) \right), 
\]
we obtain \eqref{0810_3} and complete the proof. 
\end{proof}
\begin{proposition} \label{prop_2}
Let $\phi(x)$ be a function of $C^1({\mathbb R})$. 
Assume that $\phi'(x) \ll 1$ for $|x| \leq 1$, 
$\phi'(x) \ll x^{-2}$ for $|x| \geq 1$ 
and $u \mapsto \frac{d}{du}\phi({\rm Re}\frac{\zeta'}{\zeta}(\tfrac{1}{2}+\omega+iu))$ 
is bounded on $\mathbb R$. 
We define 
\begin{equation} \label{0807_1}
\ddot{\gamma}_{n}=\varrho_\omega^{1/2}\gamma_n^{(2)} = \Bigl( \frac{\gamma_n}{2\pi}\log\frac{\gamma_n}{2\pi e} - n \Bigr)
\frac{1}{2\pi}\log\frac{\gamma_n}{2\pi e}.
\end{equation}
Then 
\begin{equation} \label{0812_2}
\aligned
\frac{1}{N_\omega(T)} \sum_{0<\gamma_n \leq T} \phi\left(\ddot{\gamma}_{n+1}-\ddot{\gamma}_n\right) 
& = \frac{1}{N_\omega(T)}\sum_{0<\gamma_n \leq T}  
\phi\left(-\frac{1}{\pi}{\rm Re}\frac{\zeta'}{\zeta}(\tfrac{1}{2}+\omega+i\gamma_n)\right) \\
& \quad 
+O\left( \frac{\log\log T}{\log T} \right)+O(E_{2,\omega}(T))
\endaligned
\end{equation}
holds for large $T>0$. 
\end{proposition}
\begin{proof}  
On the right-hand side of \eqref{0814_3}, we have 
\[
\aligned
S_\omega(\gamma_{n+1})-S_\omega(\gamma_n)
& =
\frac{1}{\pi}
{\rm Re}\frac{\zeta'}{\zeta}(\tfrac{1}{2}+\omega+i\xi_n)(\gamma_{n+1}-\gamma_n) 
\endaligned
\]
for some $\xi_n \in (\gamma_{n},\gamma_{n+1})$ by the mean value theorem. 
Therefore,  
\begin{equation} \label{0809_5}
\aligned
\Bigl(\gamma_{n+1}^{(1)} -  \gamma_{n}^{(1)}  - 1  \Bigr) & \frac{1}{2\pi}\log\frac{\gamma_n}{2\pi e} \\
= & -\frac{1}{\pi}
{\rm Re}\frac{\zeta'}{\zeta}(\tfrac{1}{2}+\omega+i\xi_n)\frac{\gamma_{n+1}-\gamma_n}{2\pi}\log\frac{\gamma_n}{2\pi e}
+ O\left( \frac{\log\gamma_n}{\gamma_n} \right) 
\endaligned
\end{equation}
by \eqref{0814_3}. 
On the other hand, we have 
\[
\ddot{\gamma}_{n+1}- \ddot{\gamma}_n 
 = \Bigl(\gamma_{n+1}^{(1)}-\gamma_{n}^{(1)}-1\Bigr)\frac{1}{2\pi}\log\frac{\gamma_n}{2\pi e}
 + 
\Bigl(\gamma_{n+1}^{(1)}-(n+1)\Bigr)\frac{1}{2\pi}\left( \log\frac{\gamma_{n+1}}{2\pi e}-\log\frac{\gamma_n}{2\pi e} \right)
\]
by definitions \eqref{0814_2} and \eqref{0807_1}. 
The second term of the right-hand side is estimated as 
\[
\aligned
\Bigl(&\gamma_{n+1}^{(1)}  -(n+1)\Bigr)
\frac{1}{2\pi}\left( \log\frac{\gamma_{n+1}}{2\pi e}-\log\frac{\gamma_n}{2\pi e}
\right) 
= \Bigl(\gamma_{n+1}^{(1)} -(n+1)\Bigr)
\frac{1}{2\pi}
 \log\left( 1 + \frac{\gamma_{n+1}-\gamma_n}{\gamma_n} \right) \\
& \quad \ll 
n \frac{\log\log n}{\log n} \cdot \frac{\gamma_{n+1}-\gamma_n}{\gamma_n} 
\ll \gamma_n \log\log \gamma_n \cdot \frac{1}{\gamma_n\log \gamma_n}
= \frac{\log\log \gamma_n}{\log \gamma_n}
\endaligned
\]
by \eqref{0529_1}, \eqref{0808_1} and \eqref{0809_2}. 

By the above argument, we get 
\begin{equation} \label{0809_6}
\ddot{\gamma}_{n+1}- \ddot{\gamma}_n  \\
=\Bigl(\gamma_{n+1}^{(1)} - \gamma_{n}^{(1)} -1 \Bigr) \frac{1}{2\pi}\log\frac{\gamma_{n}}{2\pi e}
+ O\Bigl( \frac{\log\log \gamma_n}{\log \gamma_n} \Bigr).
\end{equation}
Combining \eqref{0809_5} and \eqref{0809_6}, we obtain 
\begin{equation} \label{0821_4}
\ddot{\gamma}_{n+1}- \ddot{\gamma}_n  \\
=-\frac{1}{\pi}
{\rm Re}\frac{\zeta'}{\zeta}(\tfrac{1}{2}+\omega+i\xi_n)\frac{\gamma_{n+1}-\gamma_n}{2\pi}\log\frac{\gamma_n}{2\pi e}
+ O\Bigl( \frac{\log\log \gamma_n}{\log \gamma_n} \Bigr)
\end{equation}
for some $\xi _n\in (\gamma_n,\gamma_{n+1})$. 
Therefore, 
\[
\aligned
\phi\left(\ddot{\gamma}_{n+1}-\ddot{\gamma}_n\right)
& =\phi\left(-\frac{1}{\pi}{\rm Re}\frac{\zeta'}{\zeta}(\tfrac{1}{2}+\omega+i\xi_n)
\frac{\gamma_{n+1}-\gamma_{n}}{2\pi}\log\frac{\gamma_n}{2\pi e}+O\left(\frac{\log\log\gamma_n}{\log \gamma_n}\right) \right) \\
& =\phi\left(-\frac{1}{\pi}{\rm Re}\frac{\zeta'}{\zeta}(\tfrac{1}{2}+\omega+i\xi_n)
\frac{\gamma_{n+1}-\gamma_{n}}{2\pi}\log\frac{\gamma_n}{2\pi e} \right) +O\left(\frac{\log\log\gamma_n}{\log \gamma_n}\right) \\
& =\phi\left(-\frac{1}{\pi}{\rm Re}\frac{\zeta'}{\zeta}(\tfrac{1}{2}+\omega+i\xi_n)
\Bigl(1+E_{2,\omega}(\gamma_n)\Bigr) \right) +O\left(\frac{\log\log\gamma_n}{\log \gamma_n}\right)
\endaligned
\]
by the mean value theorem and \eqref{0809_2}, since $\phi'(x)$ is bounded. 

Now we take $T_0>0$ so that the size of the error term $O(E_{2,\omega}(t))$ of Lemma \ref{0812_1} is less than $1/2$ for every $t \geq T_0$.  
We put $r(t)=-{\rm Re}\frac{\zeta'}{\zeta}(\tfrac{1}{2}+\omega+it)$, 
$I_1(T)=\{t \in [T_0,T] \,:\, |r(t)| \leq 2/3 \}$ and $I_2(T)=\{t \in [T_0,T] \,:\, |r(t)| > 2/3 \}$ 
so that $[T_0,T]=I_1(T) \cup I_2(T)$. 

If $\gamma_n \geq T_0$ and $\xi_n \in I_1(T)$, we have 
\[
\aligned
\phi\Bigl(r(\xi_n)(1+O(E_{2,\omega}(\gamma_n)))\Bigr)-\phi(r(\xi_n)) 
& = \pm \int_{r(\xi_n)}^{r(\xi_n)(1+O(E_{2,\omega}(\gamma_n)))} \phi'(u)du  \\
& \ll |r(\xi_n)| E_{2,\omega}(\gamma_n) \leq E_{2,\omega}(\gamma_n),
\endaligned
\]
since $|r(\xi_n)| \leq 1$ and $|r(\xi_n)(1+O(E_{2,\omega}(t)))|\leq 1$. 

If $\gamma_n \geq T_0$ and $\xi_n \in I_2(T)$, we have 
\[
\aligned
\phi\Bigl(r(\xi_n)(1+O(E_{2,\omega}(\gamma_n)))\Bigr)-\phi(r(\xi_n)) 
&= \pm \int_{r(\xi_n)}^{r(\xi_n)(1+O(E_{2,\omega}(\gamma_n)))} \phi'(u)du  \\ 
& \ll \left| \frac{E_{2,\omega}(\gamma_n)}{r(\xi_n)(1+O(E_{2,\omega}(\gamma_n)))} \right| \ll E_{2,\omega}(\gamma_n),
\endaligned
\]
since $|r(\xi_n)(1+O(E_{2,\omega}(\gamma_n)))|\geq 1/3$. Therefore, 
\[
\phi\Bigl(r(\xi_n)(1+E_{2,\omega}(\gamma_n))\Bigr)=\phi(r(\xi_n))+O(E_{2,\omega}(\gamma_n)) 
\]
for every $\gamma_n \geq T_0$ and $\xi_n \in [T_0,T]$. Moreover, we have 
\[
\phi\Bigl(r(\xi_n)(1+E_{2,\omega}(\gamma_n))\Bigr)=\phi(r(\gamma_n))+ O\left( \frac{1}{\log\gamma_n} \right)+O(E_{2,\omega}(\gamma_n)) 
\]
by the mean value theorem, since $\frac{d}{du}\phi(r(u))$ is bounded on $\mathbb R$, 
$\gamma_{n}<\xi_n<\gamma_n$ and $\gamma_{n+1}-\gamma_{n}\ll (\log\gamma_n)^{-1}$.
Therefore, 
\[
\aligned
\frac{1}{N_\omega(T)} \sum_{0<\gamma_n \leq T} & \phi\left(\ddot{\gamma}_{n+1}-\ddot{\gamma}_n\right) 
 =\frac{1}{N_\omega(T)}\sum_{0<\gamma_n \leq T} 
\phi\left(-\frac{1}{\pi}{\rm Re}\frac{\zeta'}{\zeta}(\tfrac{1}{2}+\omega+i\gamma_n) \right) \\
& 
+O\left(\frac{1}{N_\omega(T)}\sum_{0<\gamma_n \leq T} \frac{\log\log \gamma_n}{\log \gamma_n} \right)
+O\left(\frac{1}{N_\omega(T)}\sum_{0<\gamma_n \leq T} E_{2,\omega}(\gamma_n) \right).
\endaligned
\] 
By the Stietjes integral and integration by parts, we have 
\[
\sum_{0<\gamma_n \leq T} \frac{\log\log \gamma_n}{\log \gamma_n}
 = \int_{\gamma_1}^{T} \frac{\log\log t}{\log t} dN_\omega(t) 
 \ll \int_{\gamma_1}^{T} \frac{\log\log t}{\log t} (\log t) \, dt 
\ll  T\log\log T
\]
and
\[
\sum_{0<\gamma_n \leq T} E_{2,\omega}(\gamma_n)
 = \int_{\gamma_1}^{T} E_{2,\omega}(t) dN_\omega(t) 
 \ll \int_{\gamma_1}^{T} E_{2,\omega}(t) (\log t) \, dt 
 \ll N_\omega(T) E_{2,\omega}(T). 
\]
Hence we obtain \eqref{0812_2}. 
\end{proof}

\subsection{Proof of Theorem \ref{thm_01} under RH} 

Put $\sigma=1/2+\omega$. By Proposition \ref{prop_1} and \ref{prop_2},
\begin{equation} \label{0817_1}
\aligned
\frac{1}{N_\omega(T)} \sum_{0<\gamma_n \leq T} \phi\left(\ddot{\gamma}_{n+1}-\ddot{\gamma}_n\right) 
&= \frac{1}{2T}\int_{-T}^{T} \phi\left(-\frac{1}{\pi}{\rm Re}\frac{\zeta'}{\zeta}(\sigma+it)\right) \, dt  \\
&\quad +O\left( \frac{\log\log T}{\log T}\right)+O(E_{2,\omega}(T))
\endaligned
\end{equation}
holds for large $T>0$, 
since ${\rm Re}(\zeta'/\zeta)(\sigma+it)$ is an even function of $t \in \mathbb R$. 

For any continuous and bounded function $\phi(x)$ on $\mathbb R$, $\phi({\rm Re}(z))$ is a continuous and bounded function on $\mathbb C$, 
because $z \mapsto \frac{1}{2}(z+\bar{z})$ is a continuous function from $\mathbb C$ into $\mathbb R$. 
Therefore, by applying  formula \eqref{formula_IM} to $\Phi(z)=\phi(-\tfrac{1}{\pi}{\rm Re}(z))$, 
we have 
\[
\lim_{T\to\infty}\frac{1}{2T}\int_{-T}^{T}\phi\left( -\frac{1}{\pi}{\rm Re}\frac{\zeta'}{\zeta}(\sigma+it) \right) \, dt
= \frac{1}{2\pi}\int_{-\infty}^{\infty}\pi m_\sigma(\pi u) \phi(u) \, du,
\]
since the $m$-function $m_\sigma(u)$ of \eqref{formula_IM2} is even. 
Hence we obtain 
\[
\lim_{T \to \infty}\frac{1}{N_\omega(T)} \sum_{0<\gamma_n \leq T} \phi\left(\ddot{\gamma}_{n+1}-\ddot{\gamma}_n\right) 
= \frac{1}{2\pi}\int_{-\infty}^{\infty}\pi m_\sigma(\pi u) \phi(u) \, du,
\]
since $\lim_{T \to \infty} E_{2,\omega}(T)=0$ for any fixed $\omega>0$. 
This implies \eqref{0731_2} by $\gamma_n^{(2)}=\varrho_\omega^{-1/2} \ddot{\gamma}_n$. 
\hfill $\Box$

\subsection{Proof of Theorem \ref{thm_01}} 
Let $X_\omega(s)$ be $A_\omega(s)$ or $B_\omega(s)$. 
We arrange the zero $\rho=\beta+i\gamma$ of $X_\omega(s)$ with $\gamma>0$ 
in a sequence $\rho_n=\beta_n+i\gamma_n$ so that $\gamma_{n+1} \geq \gamma_n$. 
Firstly, we recall that the numbers of zeros of $X_\omega(s)$ up to height $T$ and outside the line $\sigma=1/2$ 
are bounded by $T^{1-a\omega}(\log T)^2$ for any $a < 1$ (\cite[Theorem 1]{Li09}). 
In addition, for given $0<\delta<1$ and $B>0$, 
we can take an open subset $E \subset (0,\infty)$ such that 
\begin{itemize} 
\item the measure of $[T,2T] \cap E$ is bounded by $T/(\log T)^B$ for every $T \geq 2$,
\item the number of zeros of $X_\omega(1/2+it)$ 
for $t \in [T,2T]$ is bounded by $T/(\log T)^B$ for every $T \geq 2$, 
\item the zeros of $X_\omega(1/2+it)$ for $t \in [T,2T] \setminus E$ are simple, 
\item $[\gamma_n,\gamma_{n+1}] \subset [T,2T]\setminus E$ if $\gamma_n \in [T,2T] \setminus E$, 
\item $\gamma_{n+1} - \gamma_{n} = O(1/\log T)$ if $\gamma_n \in [T,2T] \setminus E$, 
\item $S_\omega(t)$ is of $C^\infty$ class in $(0,\infty) \setminus E$, 
\item the estimate 
\begin{equation} \label{0906_1}
\frac{\zeta'}{\zeta}(\tfrac{1}{2}+\omega+it) \ll (\log T)^{1-\delta}
\end{equation}
holds for $t \subset [T,2T]\setminus E$, 
\end{itemize} 
by \cite[Theorem 1]{Li09} and the proof of \cite[Theorem 2]{Li09}. 
Therefore, we have
\[
\lim_{T \to \infty} \frac{1}{N_\omega(T)}\sum_{0<\gamma_n \leq T}f(\gamma_{n})
= \lim_{T \to \infty} \frac{1}{N_\omega(T)}\sum_{{0<\gamma_n \leq T}\atop{\gamma_n \not\in E}}f(\gamma_{n}). 
\]
Using \eqref{0906_1} instead of \eqref{0809_7} 
for a calculation of the right-hand side, 
we obtain \eqref{0810_3}, \eqref{0812_2} and \eqref{0817_1} by replacing $E_{2,\omega}(t)$ by $(\log t)^{-\delta}$ 
in a way similar to the conditional proof of Theorem \ref{thm_01}.  
Hence we obtain Theorem \ref{thm_01}. 
\hfill $\Box$

\subsection{Proof of Theorem \ref{thm_02}} 
Let $\mu_\sigma$ be the variance of $M_\sigma(z)$: 
\begin{equation} \label{0821_2}
\mu_\sigma = \frac{1}{2\pi}\int_{\mathbb C}M_\sigma(z)|z|^2\,dudv. 
\end{equation}
Then we have 
\[
\varrho_\omega =\frac{1}{2\pi^2}\,\mu_{\sigma}
\]
for $\sigma=1/2+\omega$ by \cite[(4.1.8), (4.2.1)]{IM2} or \cite[(1.2.17), (1.2.21)]{Ihara02}. 
Thus, 
by using the Fourier inversion formula
\[
M_\sigma(u+iv)  = \frac{1}{2\pi}\int_{\mathbb C}\tilde{M}_\sigma(x+iy)e^{-i(xu+yv)} dxdy,
\]
we obtain
\[
\aligned
\lim_{\omega \to 0^+} \pi \varrho_{\omega}^{1/2}m_{\frac{1}{2}+\omega}(\pi \varrho_{\omega}^{1/2}u) 
& = \lim_{\omega \to 0^+} \frac{1}{2}\int_{-\infty}^{\infty}
\mu_{\sigma}M_{\sigma}(\mu_{\sigma}^{1/2}\frac{u+iv}{\sqrt{2}}) \, dv \\
& = \lim_{\omega \to 0^+} \int_{-\infty}^{\infty}\tilde{M}_\sigma(\sqrt{2}{\mu_{\sigma}^{-1/2}}x)e^{-ix u}\, dx.
\endaligned
\]
The integrand of the right-hand side is estimated as 
\[
|\tilde{M}_\sigma(\sqrt{2}{\mu_{\sigma}^{-1/2}}z)|
\leq \exp(-\sqrt{2}|z|/8)
\]
as well as \cite[(2.4.2)]{Ihara02} if $\sigma$ is sufficiently close to $1/2$. 
Therefore, by applying Lebesgue's convergence theorem to the right-hand side together with 
\[
\lim_{\sigma \to 1/2}\tilde{M}_\sigma(\mu_{\sigma}^{-1/2}z)=\exp(-|z|^2/4),  
\]
which is a special case of \cite[Lemma A]{Ihara02},
we obtain 
\[
\lim_{\sigma \to 1/2}\int_{-\infty}^{\infty}\tilde{M}_\sigma(\sqrt{2}{\mu_{\sigma}^{-1/2}}x)e^{-ixu}\, dx
= \int_{-\infty}^{\infty} \exp(-x^2/2) \, e^{-ixu}\, dx
= \sqrt{2\pi} \exp\left( -\frac{u^2}{2} \right). 
\]
This implies Theorem \ref{thm_02}. \hfill $\Box$ 

\section{Concluding remarks}

Before concluding the main parts of the paper, we give several comments and remarks. 

\subsection{On the range of test functions} 
In order to extend the range of test functions of formula \eqref{0731_2}, 
we need to extend the range of test functions of formula \eqref{formula_IM}. 
An optimistic expectation is that 
formula \eqref{formula_IM} 
holds for any continuous function $\Phi(z)$ on $\mathbb C$ 
or the characteristic function of either a compact subset of $\mathbb C$ or the complement of such a subset 
if we assume RH. 
However, the range of test functions of \eqref{formula_IM} could possibly be much more delicate problem. 
In fact, if we apply \eqref{formula_IM} formally to the test function $\Phi(w)=|w|^{2}$ 
together with \eqref{0821_3} below, we obtain
\[
\lim_{T \to \infty} \frac{1}{2T}\int_{-T}^{T} 
\left|\frac{\zeta'}{\zeta}(\sigma+it) \right|^2 \, dt 
= \sum_{n=1}^{\infty}\frac{\Lambda(n)^2}{n^{2\sigma}}=\mu_\sigma.
\]
This agree with the asymptotic formula 
\[
 \frac{1}{T}\int_{0}^{T} 
\left|\frac{\zeta'}{\zeta}(\sigma+it) \right|^{2} \, dt
\sim \sum_{n=1}^{\infty}\frac{\Lambda(n)^2}{n^{2\sigma}}
\]
for $(\sigma-1/2)\log T \to \infty$ which is followed from the estimate $S(T)=O(\log T/\log\log T)$ 
of Selberg \cite[(1.2)]{Sel}, 
where $f \sim g$ means that the ratio $f/g$ tends to one. 
It is easy to see that $\mu_\sigma \sim 1/(2\sigma-1)^2$ as $\sigma \to 1/2$. 
Thus, we obtain the asymptotic formula 
\[
 \frac{1}{T}\int_{0}^{T} 
\left|\frac{\zeta'}{\zeta}\left(\frac{1}{2}+\frac{a}{\log T}+it\right) \right|^{2} \, dt
\sim \frac{1}{4a^2} (\log T)^2
\]
as $a\to \infty$ and $T \to \infty$ with $a=o(\log T)$. 
On the other hand, Goldston--Gonek--Montgomery \cite{GGM} discovered that, assuming RH,
\[
 \frac{1}{T}\int_{0}^{T} 
\left|\frac{\zeta'}{\zeta}\left(\frac{1}{2}+\frac{a}{\log T}+it\right) \right|^{2} \, dt
\sim \frac{1-e^{-2a}}{4a^2} (\log T)^2
\]
as $T \to \infty$ for any fixed $a>0$ 
is equivalent to Montgomery--Odlyzko conjecture. 
The above facts do not contradict each other, 
but they suggest a need of careful consideration 
for the range of test functions when $\sigma$ close to $1/2$. 

\subsection{On the second normalization}
Applying \eqref{formula_IM} formally to the test function $\Phi(w)=|{\rm Re}(w)|^{2}=w^2+2w\bar{w}+\bar{w}^2$ 
together with \eqref{0821_3} below, we have
\[
\lim_{T \to \infty} \frac{1}{2T}\int_{-T}^{T} 
\left|{\rm Re} \frac{\zeta'}{\zeta}(\sigma+it) \right|^2 \, dt 
= \frac{1}{2} \sum_{n=1}^{\infty}\frac{\Lambda(n)^2}{n^{2\sigma}}.
\]
Therefore, by \eqref{0529_1}, \eqref{0808_1} and \eqref{0821_4}, we obtain 
\[
|\ddot{\gamma_{n+1}}-\ddot{\gamma_n}| \approx \left| \frac{1}{\pi}{\rm Re} \frac{\zeta'}{\zeta}(\sigma+i\xi_n) \right| 
\approx 
\varrho_{\sigma-1/2}
\]
on average in spite of \eqref{0809_7}. This is a reason on the normalizing factor $\varrho_\omega^{-1/2}$ of \eqref{0731_1}. 
The factor $(1/(2\pi))\log(\gamma_n/(2\pi e))$ of \eqref{0731_1} 
is a kind of technical adjustment to establish a bridge between 
the nearest-neighbour spacing of normalized zeros and the $M$-function.

\subsection{On a relation with Montgomery-Odlyzko conjecture} 
The functions $A_\omega(s)$ and $B_\omega(s)$ 
are holomorphic in $(\omega,s)$ as a function of two complex variables, 
and all their zeros are simple under RH if $\omega$ is a nonzero real number. 
Hence the sets of imaginary parts of $n$th zeros $ \{ \gamma_n(\omega)\,|\, \omega>0\}$ make 
analytic loci in $(0,\infty) \times (0,\infty)$, 
and they do not intersect each other. 
Moreover, assuming the simplicity of zeros of $\xi(s)$, 
$\lim_{\omega \to 0}\gamma_{n+1}(\omega)\not=\lim_{\omega \to 0}\gamma_n(\omega)$ 
for each $n \geq 1$. 
Therefore, we expect that the distribution of $\gamma_{n+1}^{(1)}(\omega)-\gamma_{n}^{(1)}(\omega)$ 
approximates well 
the distribution of the nearest-neighbor spacings $\gamma_{n+1}^{(1)}(0)-\gamma_{n}^{(1)}(0)$ 
if $\omega>0$ is small enough. 
In this sense, 
the distribution of 
$\ddot{\gamma}_{n+1}(\omega)-\ddot{\gamma}_{n}(\omega)$ 
should approximate
the distribution of 
$\gamma_{n+1}^{(1)}(0)-\gamma_{n}^{(1)}(0)-1$ 
up to a correction factor, since 
\[
\ddot{\gamma}_{n+1}(\omega)- \ddot{\gamma}_n(\omega)  \\
~\sim~\Bigl(\gamma_{n+1}^{(1)}(\omega) - \gamma_{n}^{(1)}(\omega) -1 \Bigr) \frac{1}{2\pi}\log\frac{\gamma_{n}(\omega)}{2\pi e}
\]
for large $n$ by \eqref{0809_6}. Moreover, we have 
\[
\gamma_{n+1}^{(2)}(\omega)- \gamma_n^{(2)}(\omega)  \\
~\sim~\Bigl(\gamma_{n+1}^{(1)}(\omega) - \gamma_{n}^{(1)}(\omega) -1 \Bigr) 
\]
when $\sqrt{2}\,\omega\log\gamma_{n}(\omega) \sim 1$ as $\omega \to 0^+$, 
since $\rho_\omega \sim 1/(8\pi^2 \omega^2)$ as $\omega \to 0^+$. 
Therefore, for small $\omega>0$, the distribution of $\gamma_{n+1}^{(2)}(\omega)-\gamma_{n}^{(2)}(\omega)$ 
around the height $\exp(1/\omega)$ 
approximate the $-1$ shift of the nearest-neighbor spacing distribution of $\gamma_{n+1}^{(1)}(0)-\gamma_{n}^{(1)}(0)-1$ 
in the same range. Conversely, the distribution of $\gamma_{n+1}^{(1)}(0)-\gamma_{n}^{(1)}(0)-1$ 
around a height $T>0$ is approximated by the distribution of $\gamma_{n+1}^{(2)}(\omega)-\gamma_{n}^{(2)}(\omega)$ 
for $\omega \sim 1/(\sqrt{2}\log T)$. 

However, the limit of the density function in Theorem \ref{thm_02} 
is quite different from 
a shift of the density function
\[
p(u) \approx \frac{32}{\pi^2}u^2\exp\left(-\frac{4}{\pi}u^2 \right) 
\]
of the nearest-neighbour spacing distribution for GUE 
predicted in the Montgomery-Odlyzko conjecture. 
In order to fill this gap, 
we may need a detailed study of the second error term of \eqref{0817_1}, which tends to $O(1)$ as $\omega \to 0^+$, 
and the effect of the normalizing factor $\varrho_\omega$ of \eqref{0731_1}. 

\subsection{On possible generalization} 
Let $L(s,f)$ be a self-dual $L$-function in a sense of 
Iwaniec--Kowalski \cite[Chap. 5]{IK} 
which includes Dedekind zeta-functions, 
Dirichlet $L$-functions associated to real primitive characters, 
Hecke $L$-functions associated to self-dual Hecke characters, 
automorphic $L$-functions associated to self-dual primitive holomorphic/Maass cusp forms, etc. 
For such $L$-function, a family of functions $A_\omega(s,f)$ and $B_\omega(s,f)$ corresponding to \eqref{AB} 
is defined as well, and it is established in a way similar to \cite{Jeff05} that 
the distribution of spacings of the normalized imaginary parts 
of the zeros of $A_\omega(s,f)$ and $B_\omega(s,f)$ 
converges to a limiting distribution of equal spacings of length one 
if we assume the Grand Riemann Hypothesis and the Ramanujan--Petersson conjecture for $L(s,f)$.
A key ingredient is an analogue of \eqref{0809_7} 
and other standard analytic properties of $L$-functions (see \cite[Chap. 5]{IK}). 
Therefore, an analogue of the second normalization \eqref{0731_1} is defined as well. 

However, an analogue of the $M$-function $M_\sigma(z)$ 
is not known except for the case of Dedekind zeta functions. 
It is an interesting problem to find an analogue of the function $M_\sigma(z)$ 
for $L(s,f)$, but it is not obvious what it is, 
even if it may not be hard to find an analogue of $M_\sigma(z)$ 
by a way similar to \cite{Ihara01} 
for degree one $L$-functions like Dirichlet/Hecke $L$-functions 
for real/self-dual characters. 

\subsection*{Acknowledgements}
The author would like to thank Lin Weng and Kohji Matsumoto for their interests and valuable comments on this work.   
The author is supported by KAKENHI (Grant-in-Aid for Young Scientists (B)) No. 25800007. 

\appendix

\section{$M$-function} 

In this part, we review a construction and basic properties of the $M$-function $M_\sigma(z)$ in formula \eqref{formula_IM} 
according to Ihara \cite{Ihara01, Ihara02} and Ihara--Matsumoto \cite{IM2}. 
See these references for details. 
\medskip

Let $\Lambda:{\mathbb N} \to {\mathbb R}$ be the von Mangoldt function, that is, 
$\Lambda(n)=\log p$ if $n=p^k$ for some prime number $p$ and integer $k \geq 1$, 
and $\Lambda(n)=0$ otherwise. 
We define arithmetic functions $\Lambda_k:{\mathbb N} \to {\mathbb R}$ by 
\[
\left( -\frac{\zeta'}{\zeta}(s) \right)^k 
= \left( \sum_{n=1}^{\infty} \frac{\Lambda(n)}{n^s}  \right)^k 
= \sum_{n=1}^{\infty} \frac{\Lambda_k(n)}{n^s} 
\]
for $k \geq 1$ and $\Lambda_0(n)=1$ if $n=1$, and $\Lambda_0(n)=0$ otherwise. 
For a positive integer $n$ and $z\in \mathbb C$, we define 
\[
\lambda_z(n) = \sum_{k=0}^{\infty}(-i/2)^k \frac{\Lambda_k(n)}{k!}z^k.
\]
The series converges absolutely and uniformly on every compact subset of $\mathbb C$, 
and it is a polynomial of $z$ by \cite[(3.8.5), (3.8.6)]{Ihara01}. 
Moreover, we have 
\begin{equation} \label{0819_1}
\lambda_{z}{}(mn)=\lambda_z(m)\lambda_z(n) \quad \text{if $(m,n)=1$}
\end{equation}
(\cite[Prop. 3.8.11(i)]{Ihara01}). For a prime number $p$ and complex numbers $s,z \in \mathbb C$, 
we define
\[
\tilde{M}_{s,p}(z)=\sum_{j=0}^{\infty} \frac{\lambda_{p^j}(z)\lambda_{p^j}(\bar{z})}{p^{2js}}
\]
The series converges absolutely for all $s$ with ${\rm Re}(s)>0$ 
and $z$ in a compact subset of $\mathbb C$ by \cite[prop. 3.9.4(i)]{Ihara01}. 
Using $\tilde{M}_{s,p}(z)$, we define $\tilde{M}_s(z)$ by the Euler product 
\begin{equation} \label{0821_1}
\tilde{M}_{s}(z) = \prod_{p}\tilde{M}_{s,p}(z),
\end{equation}
where $p$ runs over all prime numbers. 
The product converges for all $s$ with ${\rm Re}(s)>1/2$ 
and $z$ in a compact subset of $\mathbb C$ (\cite[Theorem 5]{Ihara01}). 
We have the Dirichlet series expansion 
\[
\tilde{M}_{s}(z)=\sum_{n=1}^{\infty} \frac{\lambda_{z}(n)\lambda_{\bar{z}}(n)}{n^{2s}}
\]
by \eqref{0819_1} 
and the series on the right-hand side converges absolutely all $s$ with ${\rm Re}(s)>1/2$ 
and $z$ in a compact subset of $\mathbb C$ by \cite[prop. 3.9.4(ii)]{Ihara01}. 

For $\sigma>1/2$ and $z \in \mathbb C$, $\tilde{M}_\sigma(z)$ is a real analytic function of $\sigma$ and $z$ 
which does not vanish identically, 
and 
satisfy $\tilde{M}_\sigma(z)=\tilde{M}_\sigma(\bar{z})=\overline{\tilde{M}_\sigma(-\bar{z})}$ 
and 
$\tilde{M}_\sigma(z)=O((1+|z|)^{-n})$ for any $n \geq 1$. 
The $M$-function in formula \eqref{formula_IM} is defined by the Fourier transform 
\[
M_\sigma(z) = \frac{1}{2\pi} \int_{\mathbb C}\tilde{M}_\sigma(w) \psi_{-z}(w) \, dw, 
\]
where $\psi_z(w)=\exp(i\cdot{\rm Re}(\bar{z}w))$. 
In addition, the $M$-function is real valued, decays rapidly as $|z| \to \infty$, and the Fourier inversion formula  
\[
\tilde{M}_\sigma(z) = \frac{1}{2\pi}\int_{\mathbb C}M_\sigma(w) \psi_z(w) \, dw
\]
holds with $\tilde{M}_\sigma(0)=1$ (\cite[Theorem 2 and 3, Remark 3.4.6]{Ihara01}). 
In particular, $(2\pi)^{-1}M_\sigma(w)dw$ is a probabilistic measure on $\mathbb C$. 
Corresponding to the Euler product \eqref{0821_1}, 
the $M$-function has a convolution Euler product 
whose $p$-factor being a certain distribution. 

We have 
\begin{equation} \label{0821_3}
\frac{1}{2\pi}\int_{\mathbb C}w^a\bar{w}^{b}M_\sigma(w) \, dw = \sum_{n=1}^{\infty}\frac{\Lambda_a(n)\Lambda_b(n)}{n^{2\sigma}}
\end{equation}
unconditionally 
together with the absolute convergence of the series 
if $\sigma >1$ (\cite[Theorem 6]{Ihara01}). 
Moreover, we have the limit formula
\[
\lim_{\sigma \to 1/2} \mu_\sigma M_\sigma(\mu_\sigma^{1/2} z) = 2 e^{-|z|^2}
\]
and the convergence is uniform on $|z| \leq R$ for any $R>0$,  
where $\mu_\sigma$ is the variance in \eqref{0821_2} (\cite[Theorem 2]{Ihara02}). 
Theorem 2 is a formal consequence of this formula.  
\medskip

Historically, formula \eqref{formula_IM} 
was obtained first in 1936 by Kershner--Wintner~\cite{KW37} for $\sigma>1/2$ 
in terms of asymptotic distribution functions 
as an analogue of a work of Jessen--Wintner for $\log \zeta(s)$ in 1935. 
However, they did not explicitly give the density function.  
The density function $M_\sigma(z)$ was constructed in 1937 by Kampen-Wintner~\cite{vKW37} for $\sigma>1$ 
as an infinite convolution Euler product. 
After that formula \eqref{formula_IM} was rediscovered by Guo~\cite{Guo} in 1993. 
He constructed $M_\sigma(z)$ for  $\sigma>1/2$ as the Fourier transform of the Euler product $\prod_p \tilde{M}_{\sigma,p}(z)$ 
but test functions in \eqref{formula_IM} are restricted to smooth and compactly supported functions. 
This restriction for the test functions was relaxed to a wider class of functions by Ihara--Matsumoto \cite{IM} in 2011 
which was a goal of a series of collaboration works of Ihara and Matsumoto standing on Ihara~\cite{Ihara01}. 
In 2008, Ihara~\cite{Ihara01} studied analytic and arithmetic properties of $M_\sigma(z)$ and $\tilde{M}_\sigma(s)$ systematically and in detail 
for $\sigma>1/2$ motivated by a study on Euler-Kronecker constants of global fields. 
This work was refined in Ihara \cite{Ihara02}. 
The formulation of \eqref{formula_IM} in the introduction depends on \cite[Theorem 6]{Ihara01} and \cite{IM}.
%


\begin{thebibliography}{99}
%
\bibitem{GGM}
D. A. Goldston, S. M. Gonek, H. L. Montgomery, 
\newblock{Mean values of the logarithmic derivative of the Riemann zeta-function with applications to primes in short intervals},
\newblock{\it J. Reine Angew. Math.} 
\newblock{{\bf 537} (2001), 105--126}.  
%
\bibitem{Guo}
C. R. Guo,
\newblock{The distribution of the logarithmic derivative of the Riemann zeta function},
\newblock{Proc. London Math. Soc. (3) {\bf 72} (1996), 1--27}. 
%
\bibitem{Ihara01}
Y. Ihara, 
\newblock{On ``$M$-functions'' closely related to the distribution of $L'/L$-values},
\newblock{Publ. RIMS, Kyoto Univ. {\bf 44} (2008), no. 3, 893--954}. 
%
\bibitem{Ihara02}
Y. Ihara, 
\newblock{On certain arithmetic functions $\tilde{M}(s; z_1, z_2)$ associated with global fields: Analytic properties},
\newblock{Publ. RIMS, Kyoto Univ. {\bf 47} (2011), no. 1, 257--305}. 
%
\bibitem{IM2}
Y. Ihara, K. Matsumoto, 
\newblock{On $\log L$ and $L'/L$ for $L$-functions and the associated``$M$-functions'': Connections in optimal cases},
\newblock{\it Moscow Math. J.}
\newblock{{\bf 11} (2011), 73--111}. 
%
\bibitem{IM}
Y. Ihara, K. Matsumoto, 
\newblock{On the value-distribution of logarithmic derivatives of Dirichlet $L$-functions},
\newblock{Analytic Number Theory, Approximation Theory, and Special Functions (2014), 79--91}. 
%
\bibitem{IK}
H. Iwaniec, E. Kowalski, 
\newblock{Analytic number theory},
\newblock{American Mathematical Society Colloquium Publications, 53}, 
\newblock{\it American Mathematical Society, Providence, RI}, 2004.
%
\bibitem{vKW37}
E. R. van Kampen, A. Wintner,
\newblock{Convolutions of Distributions on Convex Curves and the Riemann Zeta Function},
\newblock{Amer. J. Math. {\bf 59} (1937), No. 1, 175--204}.
%
\bibitem{KW37}
R. Kershner, A. Wintner, 
\newblock{On the asymptotic distribution of $\zeta'/\zeta(s)$ in the critical strip},
\newblock{Amer. J. Math. {\bf 59} (1937), No. 1, 673--678}. 
%
\bibitem{Jeff05}
J. C. Lagarias,
\newblock{Zero spacing distributions for differenced L-functions},
\newblock{Acta Arith. {\bf 120} (2005), No. 2, 159--184}. 
%
\bibitem{MV}
H. L. Montgomery, R. C. Vaughan,  
\newblock{Multiplicative number theory. I. Classical theory},
\newblock{Cambridge Studies in Advanced Mathematics 97}, 
\newblock{\it Cambridge University Press, Cambridge}, 2007.
%
\bibitem{Li09}
Xiannan Li,
\newblock{Variation of the argument of the Riemann $\xi$ function on vertical line},
\newblock{Acta Arith. {\bf 137} (2009), No. 3, 277--284}. 
%
\bibitem{Sel}
A. Selberg,
\newblock{On the remainder in the formula for $N(T)$, the number of zeros of $\zeta(s)$ in the strip $0<t<T$}, 
\newblock{\it Avh. Norske Vid. Akad. Oslo. I.} 
\newblock{{\bf 1944}, (1944). no. 1, 1--27}.
%
\bibitem{Tit}
E.C. Titchmarsh,
\newblock{The Theory of the Riemann Zeta-Function},  
\newblock{Second edition, Edited and with a preface by D. R. Heath-Brown},  
\newblock{\it The Clarendon Press, Oxford University Press, New York}, 
\newblock{1986}.
%
\bibitem{We1}
L. Weng,
\newblock{Distributions of Zeros of Zeta Functions}, 
\newblock{a personal communication, Apr., 2014}.
%
\bibitem{We2}
L. Weng,
\newblock{Distributions of Zeros for Non-Abelian Zeta Functions}, 
\newblock{preprints, Sept., 2014}, 
\newblock{available at \url{http://www2.math.kyushu-u.ac.jp/~weng/zetazeros.pdf}} 
%
\bibitem{WZ}
L. Weng, D. Zagier, 
\newblock{Higher rank zeta functions and Riemann hypothesis for elliptic curves}, 
\newblock{preprints, Sept., 2013}, 
\newblock{available at \url{http://www2.math.kyushu-u.ac.jp/~weng/ECRH.pdf}} 
\end{thebibliography}
\end{document}